\documentclass[10pt,a4paper]{article}
\usepackage[utf8]{inputenc}
\usepackage{amsmath}
\usepackage{amsfonts}
\usepackage{amssymb}
\usepackage{amsthm}
\usepackage{float}
\usepackage{aliascnt}
\usepackage{graphicx} % to inclde figures
\usepackage{enumitem}   
\newtheorem{lemma}{Lemma}

\newtheorem*{theorem}{Theorem}

\newtheorem*{remark}{Remark}

\title{Levin's conjecture for an equation of length nine
}
\author{Muhammad Saeed Akram$^{1,*}$, Khawar Hussain$^{2}$}

\date{$^{1}$Department of Mathematics, Ghazi University, Faculty of Science, Dera Ghazi Khan 32200, Pakistan\\
$^{2}$Department of Mathematics, Khwaja Fareed University of
Engineering and Information Technology,
Rahim Yar Khan 64200, Pakistan\\
$^{*}$Correspondence: mrsaeedakram@gmail.com}
\begin{document}
\maketitle
\noindent \textbf{Abstract.}
  Levin's conjecture has been established to hold for group equations of length up to seven. Recently, it is shown that Levin's conjecture is also true (modulo exceptional cases) for some group equations of length eight and nine.
In this paper, we consider a group equation of length nine and show that Levin's conjecture is true for this equation modulo some exceptional cases.

%The Levin conjecture is valid for this equation modulo exceptional cases. The Levin conjecture is shown to hold for group equation of length nine modulo some exceptional cases. The weight test and curvature distribution method are used to determine the validity of Levin conjecture. The  conjecture was prposed by Levin in 1962 which conjectures the solvability of group equations with coefficients in a torsion free group. 
%
%In [2], Levin conjectured that every equation has solution is torsion-free. In [11], authors showed that a length eight equation over torsion-free groups is solvable. In this paper we prove that the length 9 equation $s(t)=atbtctdtet^{-1}ftgthtit^{-1}=1$ has solution over torsion-free groups.\\
\section{Introduction:}

Let $G$ be any group and choose an element $t$ distinct from $G$. A
\emph{group equation on $G$} is a word in free product $G \ast
\langle t \rangle$, where $\langle t \rangle$ is a free group
generated by $t$. That is, a group equation on $G$ is an equation of
the form $s(t)= 1$ where
\begin{equation*}
s(t) = d_1t^{\epsilon_1}d_2t^{\epsilon_2}\dots d_nt^{\epsilon_n},\
d_i \in G,\ \epsilon_i = \pm 1
\end{equation*}
and $d_i \neq 1$ if $\epsilon_i = -\epsilon_{i+1}$ for $1 \leq i
\leq {n-1}$. The group equation $s(t) = 1$ is \emph{solvable
over} $G$ if the equation $s(t) = 1$ has solution in some group
$H$ that contains $G$. That is, if there exist an injective group
homomorphism $\tau :G\longrightarrow H$ such that
$h_1h^{\epsilon_1}\dots h_nh^{\epsilon_n} = 1$ in $H$, where $h_i =
\tau(d_i)$ and $h \in H$. The \emph{length} of group equation $s(t)
= 1$ is denoted by $|s(t)|$ and defined as, the absolute sum of
power $t$, that is,
$$|s(t)|=|\epsilon_1|+|\epsilon_2|
+\dots+|\epsilon_n|.$$ The group equation $s(t) = 1$ is said to
be \emph{singular} if $\sum \epsilon_i=0$, otherwise, it is called 
\emph{non-singular}.
%, that is, $$\sum \epsilon_i \neq 0.$$
More details regarding group equations can be found in \cite{LS}.

The problem of the solvability of group equation is motivated by the
solvability of polynomial equation of degree $n$ over a field $F$.
It is well known that a polynomial equation
$b_0+b_1x+b_2x^2+\dots+b_nx^n=0$ of degree $n$  with coefficients
$b_i$ from $F$ always has a solution in a suitable extension $K$ of
degree at most $n$ over $F$.
%The solvability of group equation dates back to the solvability of polynomial equation of degree $n$ over a field $F$. It is well known that, a polynomial equation $P(x) = b_0+b_1x+b_2x^2+\dots+b_nx^n$ of degree n over a $F$ with coefficients $b_i$ from $F.$ can always exist the solution of $P(x)$ in a suitable extension$\acute{F}$ of degree at most n over $F.$
Influenced by the solvability of polynomial equations over fields, Levin in \cite{L} studied the analogous problem for groups. Levin established the solvability of the group equation $s(t) = 1$ over arbitrary groups, with $\epsilon_i$ non-negative and not necessarily 1.
%The problem of solution of group equations in
%its current form was studied by Levin in \cite{L} where he
%established the solution of group equation \ref{group equation} over
%arbitrary group with positive powers of $t.$
In this direction, a preliminary result was already obtained by
Higman et al  in \cite{HNN}, where they have shown the solvability
of group equation $t^{-1}g_1tg_2 =1$ over any torsion free groups.
This fact helped Levin to conjecture the solvability of group
equation $s(t) = 1$ without restricting the negative powers of
$t$. This conjecture is explicitly stated here  for future reference.
\newline
\textbf{Levin conjecture}: \textit{Let group $G$ is torsion free, then any equation over $G$ is solvable}.

It is important to remark here that if we remove the restriction of torsion freeness on the group $G$ then the group equations may not be solvable. For example, if $G$ is any group that is not torsion free and we take the equation $
u(t) = c^{-1}_1tc_2t^{-1} = 1$, with $c_1,c_2\in G$ having
different orders then the equation is not solvable over $G$.

Several attempts have been made in the literature to solve Levin's conjecture by specifying the length of the group equation. For equations of length less than or equal to two, an affirmative answer to Levin's conjecture is an immediate consequence of the main result in \cite{L}. Later, J. Howie in \cite{H1} obtained an affirmative answer to Levin's conjecture for equations of length $|s(t)|=3$. Then Edjvet and Howie in \cite{EH} proved that Levin's conjecture holds for equations of length $|s(t)|=4$. For equations of length $|s(t)|\leq 5$, Evangelidou in \cite{E1} has shown that Levin's conjecture is true.
%A theorem in \cite{L}, any relative presentation written
%in the form, that is, $$P= \langle G,t|a_1ta_2t\dots a_kt =
%1\rangle$$is injective. It is always hold for $k\leq 2$.
Ivanov and Klyachko in \cite{IK} have proved the conjecture for equations of length $|s(t)|\leq 6$.  Recently, Mairaj Bibi and Edjvet in \cite{BE} have established that Levin's conjecture is true for equations of length $|s(t)|=7$.
More recently, Mairaj et al in \cite{BAIA}, have obtained an affirmative answer to Levin's conjecture for an equation of length $|s(t)|=8$ modulo one exceptional case. For equations of length 9, there are only three equations of length 9 that are open \cite{AAI, ABA}. In \cite{AAI} and \cite{ABA}, the authors have proved Levin's conjecture for an equation of length $|s(t)|=9$ modulo some exceptional cases. For more recent work in this direction, see \cite{AA, ABA2, ABS}. In this paper, we consider the following group equation of length nine
\begin{equation*}\label{problem} atbtct^{-1}dtetft^{-1}gthtit^{-1}=1
\end{equation*} and prove Levin's conjecture modulo some exceptional cases. 
\section{Methodology:}

Bogley and Pride \cite{BP} introduced the relative presentation $\mathcal{P} = \langle G,t~|~r \rangle$,
where $G$ is a torsion free group and $r$ is the set of cyclically reduced group words from $G \ast \langle t \rangle$. Let $P(G)$ be a group defined from  $\mathcal{P}$, which is the quotient
of $G \ast \langle t \rangle$ by the normal closure of $r$.  The presentation $\mathcal{P}$ is
called \emph{orientable} if there does not exist any member of $r$ which is a cyclic
permutation of its inverse. It is well known that the group equation $s(t)=1$ is solvable if the natural homomorphism $\tau: G \rightarrow P(G)$ is injective.
 Bogley and Pride in \cite{BP} proved that  the natural homomorphism $\tau: G \longrightarrow P(G)$
is injective if the presentation $P$ is aspherical and
orientable. The notion of asphericity is discussed in detail  by Bogley and Pride in \cite{BP}. In our case, we are considering the equation $s(t)=atbtct^{-1}dtetft^{-1}gthtit^{-1}$ therefore  $r$ and $x$ are both consists of a single element so $\mathcal{P}$ is always orientable \cite{BE}. 
%Therefore we have to only show that $P$ is aspherical. 
Therefore, to establish the validity of Levin's conjecture, it is only left to prove that presentation $\mathcal{P}$ is aspherical. To prove that  group equation $s(t)=1$ is aspherical, we follow the methods used by Mairaj Bibi and Edjvet in \cite{BE} in proving Levin's conjecture.
In particular, we apply weight test, curvature distribution, and the change of variable methods to establish the  asphericity of presentation $\mathcal{P}$. We refer to \cite{BP} and \cite{E3} for the detailed literature regarding the weight test and curvature distribution method.

%-----------------------------------------------------------------------------------%
All the necessary definitions concerning the weight test can be found in \cite{BP}. The weight test states that if the star graph $\Gamma$ of $\mathcal{P}$ admits an aspherical weight function $\theta$, then $\mathcal{P}$ is aspherical \cite{BP}. All the definitions related to pictures can be found in \cite{H}. The curvature distribution asserts that if $K$ is a reduced picture over $\mathcal{P}$ then by Euler (or Gauss-Bonnet) formula, the sum of the curvature of all regions of $K$ is $4 \pi$, that is, $K$ contains regions of positive curvature \cite{EH}. Then, if for each region $\Delta$ of $K$ of positive curvature $c(\Delta)$, there is a neighboring region $\widehat{\Delta}$, uniquely associated with $\Delta$, such that  $c(\widehat{\Delta}) + c(\Delta) \leq 0$, then the sum of the curvature of all regions of $K$ is non-positive, which implies that $\mathcal{P}$ is aspherical \cite{BE,B}.
\section{Main Results:}

\vskip 0.4 true cm
Let $G$ be a torsion free group. By applying the transformation $u=tb$ on $s(t)= atbtct^{-1}dtetft^{-1}gthtit^{-1}=1$, it can be assumed that $b=1$. Recall that $\mathcal{P}=\langle G,t~|~s(t)\rangle$ where $$s(t)= atbtct^{-1}dtetft^{-1}gthtit^{-1} \quad (a,c,d,f,g,i \in G\setminus \{1\}, b=1,e,h \in G ).$$
 Furthermore, it can be assumed without any loss that $G$ is not cyclic and $G=\langle a,b,c,d,e,f,g,h,i \rangle$ \cite{H1}. Suppose
that $K$ is a reduced spherical diagram over $\mathcal{P}$. Up to
cyclic permutation and inversion, the regions of $K$ are given by $\Delta$ as shown in Figure \ref{1}(i). The star graph $\Gamma$ of $\mathcal{P}$ is given by Figure \ref{1}(ii).
%\begin{figure}[H]
%\centering
%        \includegraphics[width=9cm]{1.PNG}
%        \caption{Region $\Delta$ of $K$ and star graph $\Gamma$ of $\mathcal{P}$ }
%    \label{1}
%\end{figure}

%\begin{figure}
%\centering
%\includegraphics[scale=0.5]{sg.png}
%\caption{Region $\Delta$ of $K$ and star graph $\Gamma$ of $\mathcal{P}$ }
%\label{1}
%\end{figure}

\begin{figure}[H]
\centering
        \includegraphics[width=11cm]{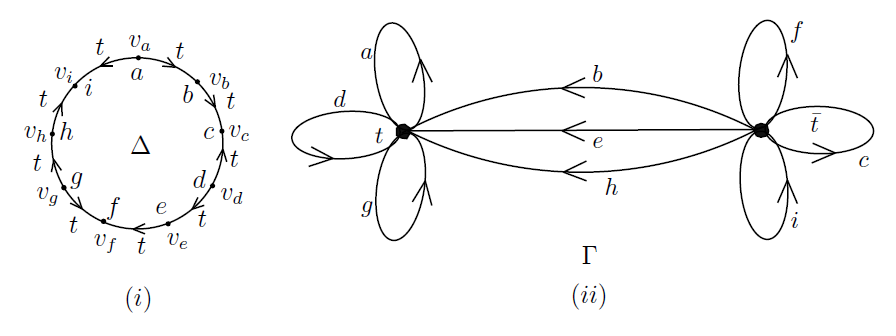}
        \caption{Region $\Delta$ of $K$ and star graph $\Gamma$ of $\mathcal{P}$ }
    \label{1}
\end{figure}
Looking at closed paths in star graph $\Gamma$, using the fact that $G$ is torsion free and working modulo cyclic permutation and inversion, the possible labels of vertices of degree 2 for a region $\Delta$ of $K$ are 
$$ S = \{ad, ad^{-1}, ag, ag^{-1}, dg, dg^{-1}, cf, cf^{-1}, ci, ci^{-1}, fi, fi^{-1}, he^{-1}, hb^{-1}, eb^{-1} \}. $$
We can work modulo equivalence, that is, modulo $t \leftrightarrow t^{-1}$, cyclic permutation, inversion, and 
$$a\leftrightarrow c^{-1}, b\leftrightarrow b^{-1}, d\leftrightarrow i^{-1}, e\leftrightarrow h^{-1}, f\leftrightarrow g^{-1}.$$
We will proceed according to the number $N$ of labels in $S$ that are admissible \cite{BP} and classify the cases correspondingly \cite{BE}. The following remark substantially reduces the number of cases to be considered.

\begin{remark}
\normalfont
The following observations hold trivially.
\begin{enumerate}%[label=(\alph*)]%[label=\roman*]

\item  If all the admissible cycle have length greater than $2$ in the region $\Delta$ then $c(\Delta)\leq c(3,3,3,3,3,3,3,3,3)= - \pi$;

\item If two admissible cycle have length $2$ and all other admissible cycles have length at least $3$ in the region $\Delta$ then $c(\Delta) \leq d(2,2,3,3,3,3,3,3,3)= -\pi/3 $;

\item If three admissible cycle have length $2$ and all other admissible cycle have length at least $3$ in the region $\Delta$ then $c(\Delta) \leq d(2,2,2,3,3,3,3,3,3)= 0$;

\item If $a=d^{-1}$ and $a=d$ are admissible then $d^2 = 1,$ a contradiction;

\item If $a=g^{-1}$ and $a=g$ are admissible then $g^2 = 1,$ a contradiction;

\item If $d=g^{-1}$ and $d=g$ are admissible then $g^2 = 1,$ a contradiction;

\item If $c=f^{-1}$ and $c=f$ are admissible then $f^2 = 1,$ a contradiction;

\item If $c=i^{-1}$ and $c=i$ are admissible then $i^2 = 1,$ a contradiction;

\item If $f=i^{-1}$ and $f=i$  are admissible then $i^2 = 1,$ a contradiction;

\item At most three of $a=d^{-1}, a=d, a=g^{-1}, a=g, d=g^{-1}, d=g$ are admissible;

\item At most three of $c=f^{-1}, c=f, c=i^{-1}, c=i, f=i^{-1}, f=i$ are admissible;

\item If any two of $a=d^{-1}, a=^{-1}g, d=g$ are admissible then so is the third;

\item If any two of $a=d^{-1}, a=g, d=g^{-1}$ are admissible then so is the third;

\item If any two of $a=d, a=g^{-1}, d=g^{-1}$ are admissible then so is the third;

\item If any two of $a=d, a=g, d=g$ are admissible then so is the third;

\item If any two of $c=i^{-1}, c=f^{-1}, f=i$ are admissible then so is the third;

\item If any two of $c=i^{-1}, c=f, f=i^{-1}$ are admissible then so is the third;

\item If any two of $c=i, c=f^{-1}, f=i^{-1}$ are admissible then so is the third;

\item If any two of $c=i, c=f, f=i$ are admissible then so is the third;

\item If any two of $h=e, h=b, e=b$ are admissible then so is the third.

\end{enumerate}
\end{remark}

%\section{\bf Cases solved by weight test}
%\vskip 0.4 true cm
In what follows, the vertex labels correspond to the closed paths in the star graph $\Gamma$. 
The proofs of the following two lemmas are the application of the weight test \cite{BP}. 
\begin{lemma} \label{lem1} The relative presentation $\mathcal{P}= \langle G,t~|~ s(t)\rangle$ is aspherical if any one of the following holds:

\begin{enumerate}

\item $a=d^{-1}, a=g^{-1}, d=g;$

\item $a=d^{-1}, a=g^{-1}, d=g$ and $R\in \{cf, cf^{-1}, ci, ci^{-1}, fi, fi^{-1}, he^{-1}, hb^{-1}, eb^{-1}\};$

\item $a=d^{-1}, a=g^{-1}, d=g, h=e$ and $R\in \{cf, cf^{-1}, ci, ci^{-1}, fi, fi^{-1}\};$

\item $a=d^{-1}, a=g^{-1}, d=g, h=b$ and $R\in \{cf, cf^{-1}, ci, ci^{-1}, fi, fi^{-1}\};$

\item $a=d^{-1}, a=g^{-1}, d=g, e=b$ and $R\in \{cf, cf^{-1}, ci, ci^{-1}, fi, fi^{-1}\};$

\item $a=d^{-1}, a=g^{-1}, d=g, c=i^{-1}, c=f^{-1}, f=i;$

\item $a=d^{-1}, a=g^{-1}, d=g, c=i^{-1}, c=f, f=i^{-1};$

\item $a=d^{-1}, a=g^{-1}, d=g, c=i, c=f^{-1}, f=i^{-1};$

\item $a=d^{-1}, a=g^{-1}, d=g, c=i, c=f, f=i;$

\item $a=d^{-1}, a=g^{-1}, d=g, h=e, h=b, e=b;$

\item $a=d^{-1}, a=g^{-1}, d=g, c=i^{-1}, c=f^{-1}, f=i$ and $R\in\{he^{-1}, hb^{-1}\};$

\item $ a=d^{-1}, a=g^{-1}, d=g, c=i^{-1}, c=f, f=i^{-1}$ and $R\in \{he^{-1}, hb^{-1}, eb^{-1}\};$

\item $a=d^{-1}, a=g^{-1}, d=g, c=i, c=f^{-1}, f=i^{-1}$ and $R\in \{ he^{-1}, hb^{-1}, eb^{-1}\};$

\item $a=d^{-1}, a=g^{-1}, d=g, c=i, c=f, f=i$ and $R\in \{he^{-1}, hb^{-1}, eb^{-1}\};$

\item $a=d^{-1}, a=g^{-1}, d=g, h=e, h=b, e=b$ and $R\in \{cf, cf^{-1}, ci, ci^{-1}, fi, fi^{-1}\};$

\item $a=d^{-1}, a=g^{-1}, d=g, c=i^{-1}, c=f^{-1}, f=i, h=e, h=b, e=b;$

\item $a=d^{-1}, a=g^{-1}, d=g, c=i^{-1}, c=f, f=i^{-1}, h=e, h=b, e=b;$

\item $a=d^{-1}, a=g^{-1}, d=g, c=i, c=f^{-1}, f=i^{-1}, h=e, h=b, e=b;$

\item $a=d^{-1}, a=g^{-1}, d=g, c=i, c=f, f=i, h=e, h=b, e=b.$

\end{enumerate}
\end{lemma}

\begin{proof}
We prove the Lemma for the case  $a=d^{-1}, a=g^{-1}, d=g$. The proof of the remaining cases follows similarly.

%\begin{figure}
%\centering
%\includegraphics[scale=0.5]{case59.png}
%\label{case59}
%\caption{Star graph $\Gamma$}
%\end{figure}

\begin{figure}[H]
\centering
        \includegraphics[width=5cm]{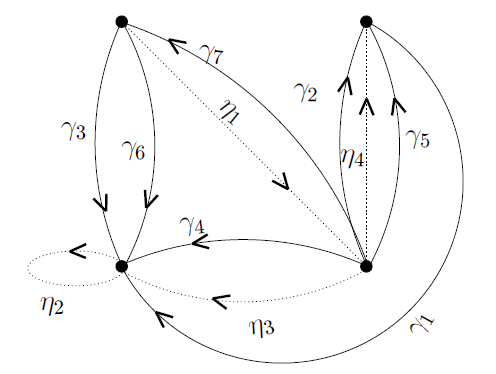}
        \caption{Star graph $\Gamma$}
    \label{case59}
\end{figure}
The relator is $s(t)=atbtct^{-1}a^{-1}tetft^{-1}a^{-1}thtit^{-1}$. We substitute $x=t^{-1}att$ to get $r_1 = xctx^{-1}etftx^{-1}hti$ and $r_2 = x^{-1}t^{-1}att$. 
We achieve the presentation $$\mathcal{Q} = \langle G, t, x~ |~r_1, r_2 \rangle.$$
The star graph $\Gamma$ for $\mathcal{Q}$ as shown in the Figure \ref{case59}, where $\gamma_1=c, \gamma_2=1, \gamma_3=e, \gamma_4=f, \gamma_5=1, \gamma_6=h, \gamma_7=i$ and $\eta_1=1, \eta_2=a, \eta_3=1, \eta_4=1$. We construct a weight function $\theta$ such that $\theta(\gamma_{1})=\theta(\gamma_{7})=\theta(\eta_{2})=\theta(\eta_{1})=0$ and we gave a weight $1$ to all other edges. Then $\sum (1- \theta(\gamma_i))=\sum (1- \theta(\eta_j))=2$ shows that the first condition of weight test is satisfied. Moreover, every cycle in the star graph having weight less than two has label $i^m$ or $a^m$, $(m \neq0)$ and $(i, a\neq1)$ and since $G$ is torsion free group so the second condition of weight test is satisfied. Furthermore, since $\theta$ assigns non-negative weights to each edge, so the third condition of weight test is obviously satisfied.

\end{proof}

\begin{lemma} \label{lem2} The relative presentation $\mathcal{P}= \langle G,t~|~ s(t)\rangle$ is aspherical if any one of the following holds:

\begin{enumerate}

\item $a=d^{-1}, a=g, d=g^{-1};$

\item $a=d^{-1}, a=g, d=g^{-1}$ and $R\in \{cf, cf^{-1}, ci, ci^{-1}, fi, fi^{-1}, he^{-1}, hb^{-1}, eb^{-1}\};$

\item $a=d^{-1}, a=g, d=g^{-1}, h=e$ and $R\in \{cf, cf^{-1}, ci, ci^{-1}, fi, fi^{-1}\};$

\item $a=d^{-1}, a=g, d=g^{-1}, h=b$ and $R\in \{cf, cf^{-1}, ci, ci^{-1}, fi, fi^{-1}\};$

\item $a=d^{-1}, a=g, d=g^{-1}, e=b$ and $R\in \{cf, cf^{-1}, ci, ci^{-1}, fi, fi^{-1}\};$

\item $a=d^{-1}, a=g, d=g^{-1}, c=i^{-1}, c=f, f=i^{-1};$

\item $a=d^{-1}, a=g, d=g^{-1}, c=i, c=f^{-1}, f=i^{-1};$

\item $a=d^{-1}, a=g, d=g^{-1}, c=i, c=f, f=i;$

\item $a=d^{-1}, a=g, d=g^{-1}, h=e, h=b, e=b;$

\item $a=d^{-1}, a=g, d=g^{-1}, c=i^{-1}, c=f, f=i^{-1}$ and $R\in \{he^{-1}, hb^{-1}\};$

\item $a=d^{-1}, a=g, d=g^{-1}, c=i, c=f^{-1}, f=i^{-1}$ and $R\in \{he^{-1}, hb^{-1}, eb^{-1}\};$

\item $a=d^{-1}, a=g, d=g^{-1}, c=i, c=f, f=i$ and  $R\{he^{-1}, hb^{-1}, eb^{-1}\};$

\item $a=d^{-1}, a=g, d=g^{-1}, h=e, h=b, e=b$ and $R\in \{cf, cf^{-1}, ci, ci^{-1}, fi, fi^{-1}\};$

\item $a=d^{-1}, a=g, d=g^{-1}, c=i^{-1}, c=f, f=i^{-1}, h=e, h=b, e=b;$

\item $a=d^{-1}, a=g, d=g^{-1}, c=i, c=f^{-1}, f=i^{-1}, h=e, h=b, e=b;$

\item $a=d^{-1}, a=g, d=g^{-1}, c=i, c=f, f=i, h=e, h=b, e=b;$

\item $a=d, a=g^{-1}, d=g^{-1};$

\item $a=d, a=g^{-1}, d=g^{-1}$ and $R\in \{cf, cf^{-1}, ci, ci^{-1}, fi, fi^{-1}, he^{-1}, hb^{-1}, eb^{-1}\};$

\item $a=d, a=g^{-1}, d=g^{-1}, h=e$ and $R\in \{cf, cf^{-1}, ci, ci^{-1}, fi, fi^{-1}\};$

\item $a=d, a=g^{-1}, d=g^{-1}, h=b$ and $R\in \{cf, cf^{-1}, ci, ci^{-1}, fi, fi^{-1}\};$

\item $a=d, a=g^{-1}, d=g^{-1}, e=b$ and $R\in \{cf, cf^{-1}, ci, ci^{-1}, fi, fi^{-1}\};$

\item $a=d, a=g^{-1}, d=g^{-1}, c=i, c=f^{-1}, f=i^{-1};$

\item $a=d, a=g^{-1}, d=g^{-1}, c=i, c=f, f=i;$

\item $a=d, a=g^{-1}, d=g^{-1}, h=e, h=b, e=b;$

\item $a=d, a=g^{-1}, d=g^{-1}, c=i, c=f^{-1}, f=i^{-1}$ and $R\in \{he^{-1}, hb^{-1}\};$

\item $a=d, a=g^{-1}, d=g^{-1}, c=i, c=f, f=i$ and $R\in \{he^{-1}, hb^{-1}, eb^{-1}\};$

\item $a=d, a=g^{-1}, d=g^{-1}, h=e, h=b, e=b$ and $R\in \{cf, cf^{-1}, ci, ci^{-1}, fi, fi^{-1}\};$

\item $a=d, a=g^{-1}, d=g^{-1}, c=i, c=f^{-1}, f=i^{-1}, h=e, h=b, e=b;$

\item $a=d, a=g^{-1}, d=g^{-1}, c=i, c=f, f=i, h=e, h=b, e=b.$

\end{enumerate}
\end{lemma}

\begin{proof}
We prove the Lemma for the case  $a=d^{-1}, a=g, d=g^{-1}$. The proof of the remaining
cases follows similarly.

%\begin{figure}
%\includegraphics[scale=0.5]{case60.png}
%%\label{case 1}
%\caption{Star graph $\Gamma$}
%\label{case60}
%\end{figure}

\begin{figure}[H]
\centering
        \includegraphics[width=5cm]{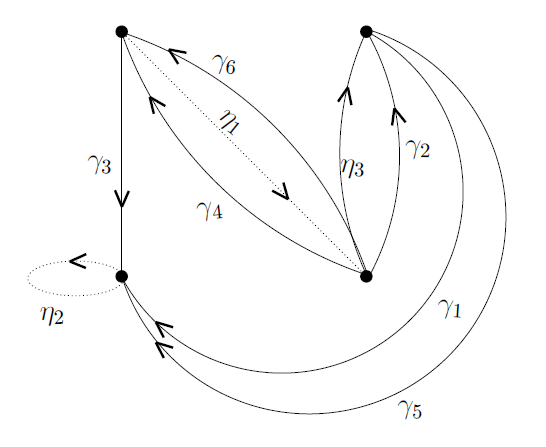}
        \caption{Star graph $\Gamma$}
    \label{case60}
\end{figure}

The relator is $s(t) = atbtct^{-1}a^{-1}tetft^{-1}athtit^{-1}$. We substitute $x = t^{-1}at$ to get $r_1 = xtcx^{-1}etfxhti$ and $r_2 = x^{-1}t^{-1}at$. We achieve the presentation $$\mathcal{Q} = \langle G, t, x~ |~r_1, r_2 \rangle.$$
The star graph for $\mathcal{Q}$ as shown in the Figure \ref{case60} where, $\gamma_1=1, \gamma_2=c, \gamma_3=e, \gamma_4=f, \gamma_5=h, \gamma_6=i$ and $\eta_1=1, \eta_2=a, \eta_3=1$. We construct a weight function $\theta$ such that $\theta(\gamma_{3})=\theta(\gamma_{2})=\theta(\eta_{3})=\theta(\eta_{2})=0$ and we gave a weight $1$ to all other edges. Then $\sum (1- \theta(\gamma_i))=\sum (1- \theta(\eta_j))=2$ shows that the first condition of the weight test is satisfied. Every cycle in the star graph having a weight less than two has label $c^m$ or $a^m$, $(m\neq0)$ and $(c, a\neq1)$,  and since $G$ is torsion free group so the second condition of the weight test is satisfied. Furthermore, since $\theta$ assigns non-negative weights to each edge, so the third condition of the weight test is satisfied.

\end{proof}

 From now onward, the label and the degree of a vertex $v$ of region $\Delta$ will be denoted by $l_{\Delta}(v)$ and $d_{\Delta}(v)$ respectively. Furthermore, $l_{\Delta}\in\{ww_1,\dots,ww_k\}$ will be indicated by $l_{\Delta}(v)= \{ww_1,\dots,ww_k\}$. The forthcoming lemmas are proved by using  curvature distribution \cite{E3}
\begin{lemma} \label{lem3} The relative presentation $\mathcal{P}= \langle G,t~|~ s(t)\rangle$ is aspherical if any one of the following holds:
\begin{enumerate}

\item $a=d^{-1}$;

\item $a=d$;

\item $d=g^{-1}$;

\item $d=g$;

\item $a=g^{-1}$;

\item $a=g$;

\item $h=e$;

\item $h=b$.
\end{enumerate}
\end{lemma}

\begin{proof}

\begin{figure}[H]
\centering
        \includegraphics[width=9cm]{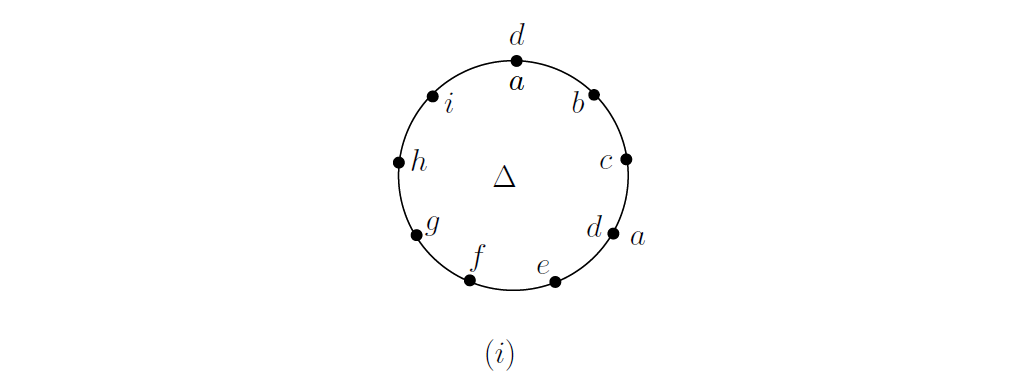}
        \caption{Region $\Delta$ }
    \label{lemm3}
\end{figure}

In all the cases, given in this Lemma there always exist only two vertices that have degree two and all other vertices have the degree at least three. So $c(\Delta)\leq -\frac{\pi}{3}.$ For demonstration, we consider the case $a=d^{-1}$.

In this case, only the vertices $v_a$ and $v_d$ have degree two and all other vertices have thee degree of at least three as shown in the Figure \ref{lemm3}(i). So $c(\Delta)\leq -\frac{\pi}{3}$, as required.

\end{proof}

\begin{lemma} \label{lem4} The relative presentation $\mathcal{P}= \langle G,t~|~ s(t)\rangle$ is aspherical if any one of the following holds:
\begin{enumerate}
\item $a=d^{-1}, c=f^{-1};$
\item  $a=d^{-1}$ and $R\in \{cf^{-1}, ci, ci^{-1}, fi, fi^{-1}, he^{-1}, hb^{-1}, eb^{-1}\};$

\item  $a=d$ and $R\in \{cf, cf^{-1}, fi, fi^{-1}, he^{-1}, hb^{-1}\};$

\item  $a=g^{-1}$ and $R\in \{cf, cf^{-1}, fi, fi^{-1}, he^{-1}, hb^{-1}, eb^{-1}\};$

\item $a=g$ and $R\in \{cf^{-1}, fi, he^{-1}, eb^{-1}\};$

\item  $d=g^{-1}$ and $R\in \{fi, fi^{-1}, he^{-1}, hb^{-1}, eb^{-1}\};$

\item $d=g$ $R\in \{fi^{-1}, hb^{-1}, eb^{-1}\};$

\item $a=d^{-1}, c=f^{-1}, e=b;$

\item $a=d^{-1}, c=i^{-1}, h=b;$

\item $a=d^{-1}, c=i, e=b;$

\item $a=g^{-1}, f=i^{-1}, h=e;$

\item $a=g^{-1}, f=i^{-1}, h=b;$

\item $a=g^{-1}, f=i, h=b;$

\item $a=g, f=i^{-1}, h=e;$

\item $d=g^{-1}, f=i^{-1}, h=e;$

\item $a=d, a=g, d=g;$

\item $h=e, h=b, e=b;$

\item $a=d, a=g, d=g$ and $R\in \{cf, ci, fi\};$

\item $h=e, h=b, e=b$ and $R\in \{ad, ag, dg\}.$
\end{enumerate}

\end{lemma}

\begin{proof}

%\begin{figure}
%\includegraphics[scale=0.5]{case.png}
%\label{case 1}
%\caption{Region $\Delta$ for case $ad, cf$}
%\label{case}
%\end{figure}

\begin{figure}[H]
\centering
        \includegraphics[width=9cm]{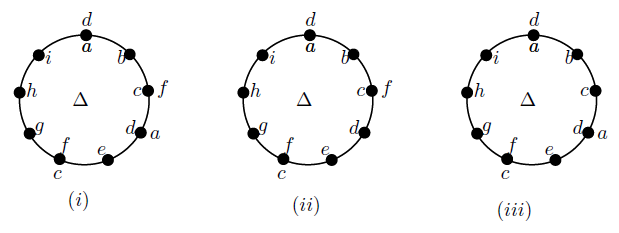}
        \caption{Region $\Delta$ }
    \label{case}
\end{figure} 

In all the cases, given in this Lemma there always exist three vertices that have degree two and all other vertices have the degree at least three. So $c(\Delta)\leq 0$. For demonstration, we consider the case $a=d^{-1} $ and $ c=f^{-1}$.

In this case, the degree of vertices $v_c$ and $v_d$ cannot be two together as shown in the Figure \ref{case}(i). So there arise two cases as shown in Figure \ref{case}(ii) and Figure \ref{case}(iii). 
\begin{enumerate}
\item $d_{\Delta}(v_a)=d_{\Delta}(v_c)=d_{\Delta}(v_f)=2;$

\item $d_{\Delta}(v_a)=d_{\Delta}(v_d)=d_{\Delta}(v_f)=2.$
\end{enumerate}

In both cases $c(\Delta)\leq 0$, as required.
\end{proof}

\begin{lemma} \label{lem5} The relative presentation $\mathcal{P}= \langle G,t~|~ s(t)\rangle$ is aspherical if any one of the following is holds:

\begin{enumerate}

\item $a=g, f=i^{-1}, e=b;$

\item $a=d^{-1}, c=i^{-1}, h=e;$

\item $a=d, f=i^{-1}, h=e;$

\item $a=d, f=i, h=e;$

\item $a=g^{-1}, c=f^{-1}, h=b;$

\item $a=d^{-1}, c=i, h=e;$

\item $a=g^{-1}, c=f, h=b;$

\item $a=d^{-1}, f=i^{-1}$ and $R\in \{he^{-1}, eb^{-1}\};$

\item $a=d^{-1}, f=i, e=b;$

\item $a=g^{-1}, f=i^{-1}, e=b;$

\item $a=g^{-1}, f=i, e=b.$

\end{enumerate}

\end{lemma}
%\subsection{$N=3;$ Case $17(iii);$  $a=g, f=i^{-1}$ and $e=b$}

\begin{proof}
We prove the Lemma for the case  $a=g, f=i^{-1}, e=b$,. The proof of the remaining cases follows similarly.

%\begin{figure}
%\includegraphics[scale=0.5]{f1.png}
%\label{case 1}
%\caption{Region $\Delta$ for case $a=g, f=i^{-1}$ and $e=b$ }
%\label{casef1}
%\end{figure}

\begin{figure}[H]
\centering
        \includegraphics[width=11cm]{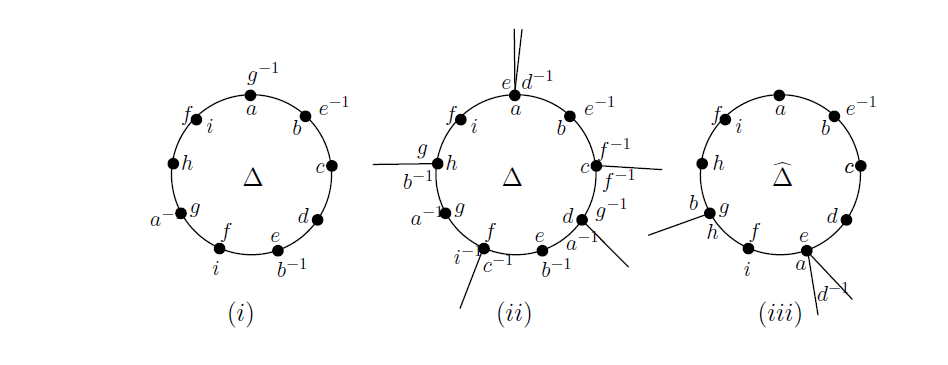}
        \caption{Regions $\Delta$ and $\widehat{\Delta}$}
    \label{casef1}
\end{figure} 

Here $l_{\Delta}(v_b)=be^{-1}, l_{\Delta}(v_g)=ga^{-1}, l_{\Delta}(v_e)=eb^{-1}, l_{\Delta}(v_f)=fi, l_{\Delta}(v_a)=ag^{-1}$ and $l_{\Delta}(v_i)=if$ as shown in the Figure \ref{casef1}(i). We consider the following case:
\begin{enumerate}

\item $d_{\Delta}(v_i)=d_{\Delta}(v_b)=d_{\Delta}(v_e)=d_{\Delta}(v_g)=2.$

\end{enumerate}

\begin{enumerate}

\item Here $l_{\Delta}(v_b)=be^{-1}$ and $l_{\Delta}(v_i)=if$ which implies $l_{\Delta}(v_a)=ead^{-1}w$, where $w\in\{h^{-1}, b^{-1}, e^{-1}\}$ so $d_{\Delta}(v_a)>3$ as shown in the Figure \ref{casef1}(ii). 

Assume $d_{\Delta}(v_f)=d_{\Delta}(v_c)=d_{\Delta}(v_d)=3$. Then $l_{\Delta}(v_f)=c^{-1}fi^{-1}$ which implies
$l_{\Delta}(v_c)=f^{-1}c\{c, f^{-1}, i, i^{-1}\}$. But $l_{\Delta}(v_c)=f^{-1}c\{i^{-1}, c\}$ implies $f^{3}=1$ or $c=1$, a contradiction. If $l_{\Delta}(v_c)=f^{-1}ci$ which forces $l_{\Delta}(v_d)=da^{-1}hw$, where $w\in \{h^{-1}, e^{-1}, b^{-1}\}$ so $d_{\Delta}(v_d)>3$, a contradiction. If $l_{\Delta}(v_c)=cf^{-2}$ then $l_{\Delta}(v_d)=da^{-1}g^{-1}.$ So curvature of $c(\Delta)\leq \frac{\pi}{6}.$ Add $c(\Delta)$ to $c(\widehat{\Delta})$ is given in the Figure \ref{casef1}(iii). Observe that in $\widehat{\Delta}$. Here, $d_{\widehat{\Delta}}(v_b)=d_{\widehat{\Delta}}(v_i)=d_{\widehat{\Delta}}(v_f)=2, d_{\widehat{\Delta}}(v_e)=4$ and all other vertices have degree atleast $3$. So $c(\widehat{\Delta)}\leq c(2,2,2,3,3,3,3,3,4)=-\frac{\pi}{6}.$ So $c(\Delta) + c(\widehat{\Delta}) \leq 0.$

\end{enumerate}
\end{proof}

\begin{lemma} \label{lem6} The relative presentation $\mathcal{P}= \langle G,t~|~ s(t)\rangle$ is aspherical if any one of the following holds:

\begin{enumerate}

\item $a=g^{-1}, c=f^{-1}, h=e;$

\item $a=d^{-1}, c=f, h=e;$

\item $a=d^{-1}, c=f^{-1}$ and $R\in \{he^{-1},hb^{-1}\};$

\item $a=d^{-1}, c=f, h=b;$

\item $a=d^{-1}, f=i^{-1}, h=b;$

\item $a=d, c=f^{-1}$ and $\{hb^{-1}, he^{-1}\};$

\item $a=d, f=i^{-1}, h=b;$

\item $a=g^{-1}, c=f, h=e;$

\item $a=g, c=f, h=e;$

\item $d=g, f=i, h=b;$

\item $d=g^{-1}, f=i^{-1}, h=b;$

\item $d=g^{-1}, f=i$ and $R\in \{eb^{-1}, hb^{-1}\};$

\item $a=d, c=f, h=e;$

\item $a=d^{-1}, f=i, h=b;$

\item $a=d,c=f, h=b;$

\item $a=d,c=i, h=b;$

\item $a=d,f=i, h=b;$

\item $a=g,f=i$ and $R\in \{he^{-1}, hb^{-1}\};$

\item $h=e, h=b, e=b, a=g^{-1}$ and $R\in \{cf, fi\};$

\item $h=e, h=b, e=b, a=d^{-1}$ and $R\in \{cf, ci, fi\};$

\item $h=e, h=b, e=b, d=g^{-1}, f=i^{-1}.$

\end{enumerate}

\end{lemma}

\begin{proof} We prove the Lemma for the cases  $a=g^{-1}, c=f^{-1}, h=e$ and $a=d^{-1}, c=f, h=e$. The proof of the remaining cases follows similarly.
\begin{enumerate}
\item First we consider the case $a=g^{-1}, c=f^{-1}, h=e$.  

%
%\begin{figure}
%\includegraphics[scale=0.5]{case37.png}
%\label{case 1}
%\caption{Region $\Delta$ for case $a=g^{-1}, c=f^{-1}$ and $h=e$ }
%\label{case37}
%\end{figure}

\begin{figure}[H]
\centering
        \includegraphics[width=11cm]{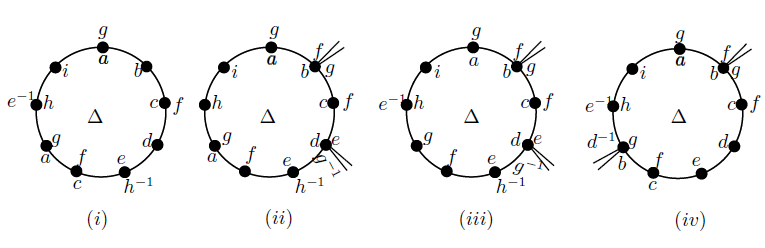}
        \caption{Regions $\Delta$ and $\widehat{\Delta}$}
    \label{case37}
\end{figure} 
Here $l(v_a)=ag, l(v_c)=cf, l(v_e)=eh^{-1}, l(v_f)=fc, l(v_g)=ga$ and $l(v_h)=he^{-1}$ as shown in the Figure  \ref{case37}(i). There are the following three
cases to consider:

\begin{enumerate}%[label=(\roman*)]
\item $d_{\Delta}(v_a)=d_{\Delta}(v_c)=d_{\Delta}(v_e)=d_{\Delta}(v_g)=2;$

\item $d_{\Delta}(v_a)=d_{\Delta}(v_c)=d_{\Delta}(v_e)=d_{\Delta}(v_h)=2;$

\item $d_{\Delta}(v_a)=d_{\Delta}(v_c)=d_{\Delta}(v_f)=d_{\Delta}(v_h)=2.$

\end{enumerate}

\begin{enumerate}% [label=(\roman*)]

\item Here $l(v_a)=ag$ and $l(v_c)=cf$ which implies  $l(v_b)=fbgw,$ where, $w\in\{e^{-1}, b^{-1}, h^{-1}\}$ so $d_{\Delta}(v_b)>3$. Same as $l(v_c)=cf$ and $l(v_e)=eh^{-1}$ which implies $l(v_d)=edg^{-1}w,$ where, $w\in\{e^{-1}, b^{-1}, h^{-1}\}$ so $d_{\Delta}(v_d)>3.$ Therefore $c(\Delta)\leq 0$ as shown in the Figure \ref{case37}(ii).

\item Here $l(v_a)=ag$ and $l(v_c)=cf$ which implies  $l(v_b)=fbgw,$ where, $w\in\{e^{-1}, b^{-1}, h^{-1}\}$ so $d_{\Delta}(v_b)>3.$ Same as $l(v_c)=cf$ and $l(v_e)=eh^{-1}$ which implies $l(v_d)=edg^{-1}w,$ where, $w\in\{e^{-1}, b^{-1}, h^{-1}\}$ so $d_{\Delta}(v_d)>3.$ Therefore $c(\Delta)\leq 0$ as shown in the Figure \ref{case37}(iii). 

\item Here $l(v_a)=ag$ and $l(v_c)=cf$ which implies  $l(v_b)=fbgw,$ where, $w\in\{e^{-1}, b^{-1}, h^{-1}\}$ so $d_{\Delta}(v_b)>3$ as shown in the Figure \ref{case37}(iv). Same as $l(v_f)=fc$ and$l(v_h)=he^{-1}$ which implies $l(v_g)=bgd^{-1}w,$ where, $w\in\{e^{-1}, b^{-1}, h^{-1}\}$ so $d_{\Delta}(v_g)>3.$ Therefore $c(\Delta)\leq 0.$
\end{enumerate}

\item Now we consider the case $a=d^{-1}, c=f, h=e$.
%\subsection{$N=3;$ case $2(i);$  $a=d^{-1}, c=f$ and $h=e:$}

%\begin{figure}
%\includegraphics[scale=0.5]{f2.png}
%\label{case 1}
%\caption{Region $\Delta$ for case $a=d^{-1}, c=f$ and $h=e$ }
%\label{casef2}
%\end{figure}

\begin{figure}[H]
\centering
        \includegraphics[width=10cm]{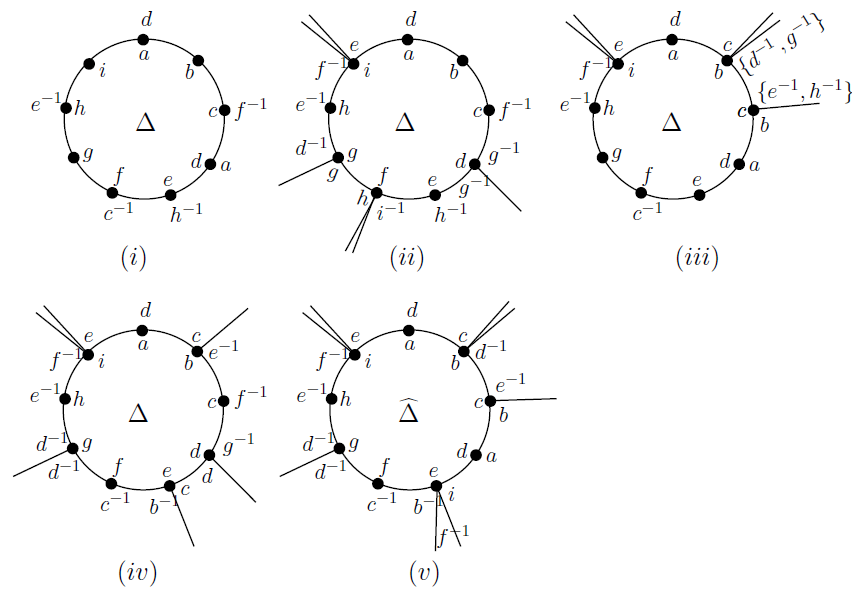}
        \caption{Regions $\Delta$ and $\widehat{\Delta}$}
    \label{casef2}
\end{figure} 

Here $l_{\Delta}(v_h)=he^{-1}, l_{\Delta}(v_a)=ad, l_{\Delta}(v_c)=cf^{-1}, l_{\Delta}(v_d)=da, l_{\Delta}(v_f)=fc^{-1}$ and $l_{\Delta}(v_e)=eh^{-1}$ as shown in the Figure \ref{casef2}(i). There are the following three cases to consider:

\begin{enumerate}

\item $d_{\Delta}(v_h)=d_{\Delta}(v_a)=d_{\Delta}(v_c)=d_{\Delta}(v_e)=2;$

\item $d_{\Delta}(v_h)=d_{\Delta}(v_a)=d_{\Delta}(v_d)=d_{\Delta}(v_f)=2;$

\item $d_{\Delta}(v_h)=d_{\Delta}(v_a)=d_{\Delta}(v_c)=d_{\Delta}(v_f)=2.$

\end{enumerate}

\begin{enumerate}

\item Here $l_{\Delta}(v_a)=ad$ and $l_{\Delta}(v_h)=he^{-1}$ which implies $l_{\Delta}(v_i)=f^{-1}iew,$ where $w\in\{e^{-1}, b^{-1}, h^{-1}\}$ so $d_{\Delta}(v_i)>3$ as shown in the Figure \ref{casef2}(ii). 

Assume $d_{\Delta}(v_d)=d_{\Delta}(v_f)=d_{\Delta}(v_g)=3$. Then $l_{\Delta}(v_d)=dg^{-2}$ which implies
$l_{\Delta}(v_g)=gd^{-1}\{d^{-1}, a, a^{-1}, g\}$. But $l_{\Delta}(v_g)=gd^{-1}\{d^{-1}, a, a^{-1}\}$ implies $d^3=1$ or $g=1$, a contradiction so $l_{\Delta}(v_g)=d^{-1}g^2$ which forces $l_{\Delta}(v_f)=i^{-1}fhw$, where $w\in \{h^{-1},e^{-1},b^{-1}\}$ which implies $d_{\Delta}(v_f)>3$, a contradiction. So degree of at least one of the vertices $v_d$, $v_f$ and $v_g$ must be greater than three. Therefore $c(\Delta)\leq 0$.

\item Here $l_{\Delta}(v_a)=ad$ and $l_{\Delta}(v_h)=he^{-1}$ which implies  $l_{\Delta}(v_i)=f^{-1}iew,$ where $w\in\{e^{-1}, b^{-1}, h^{-1}\}$ so $d_{\Delta}(v_i)>3$ as shown in the Figure \ref{casef2}(iii).

Assume $d_{\Delta}(v_c)=d_{\Delta}(v_b)=3$. Since $l_{\Delta}(v_d)=da$ which implies
$l_{\Delta}(v_c)=cb\{e^{-1}, h^{-1}, b^{-1}\}$. But $l_{\Delta}(v_c)=cbb^{-1}$ implies $c=1$ a contradiction so $l_{\Delta}(v_c)=cb\{e^{-1}, h^{-1}\}.$ If $l_{\Delta}(v_c)=cbe^{-1}$ then $l_{\Delta}(v_b)=cbd^{-1}w_1,$ where $w_1\in\{e^{-1}, h^{-1}, b^{-1}\},$ a contradiction. If $l_{\Delta}(v_c)=cbh^{-1}$ then $l_{\Delta}(v_b)=cbg^{-1}w_1$, where $w_1\in\{e^{-1}, h^{-1}, b^{-1}\},$ a contradiction. So degree of at least one of the vertices $v_c$ and $v_b$ must be greater than three. Therefore $c(\Delta)\leq 0$.

\item Here $l_{\Delta}(v_a)=ad$ and $l_{\Delta}(v_h)=he^{-1}$ which implies $l_{\Delta}(v_i)=f^{-1}iew,$ where $w\in\{e^{-1}, b^{-1}, h^{-1}\}$ so $d_{\Delta}(v_i)>3$ as shown in the Figure \ref{casef2}(iv).

Assume $d_{\Delta}(v_d)=d_{\Delta}(v_e)=d_{\Delta}(v_g)=3$. Then $l_{\Delta}(v_g)=gd^{-2}$ which implies
$l_{\Delta}(v_d)=g^{-1}d\{d, g^{-1}, a, a^{-1}\}$. But $l_{\Delta}(v_d)=g^{-1}d\{g^{-1}, a\}$ implies $g^3=1$ or $g=1$, a contradiction. If $l_{\Delta}(v_d)=g^{-1}da^{-1}$ which forces $l_{\Delta}(v_e)=b^{-1}eb^{-1}w$, where $w\in \{h, e, b\}$ which implies $d_{\Delta}(v_e)>3$, a contradiction. If $l_{\Delta}(v_d)=g^{-1}d^{2}$ then $l_{\Delta}(v_e)=ceb^{-1}.$ So curvature of $c(\Delta)=\frac{\pi}{6}$. Add $c(\Delta)$ to $c(\widehat{\Delta})$ as shown in the Figure \ref{casef2}(v). Observe in $\widehat{\Delta}$. Here, $l_{\widehat{\Delta}}(v_h)=he^{-1}$ and $l_{\widehat{\Delta}}(v_a)=ad$ implies $ l_{\widehat{\Delta}}(v_i)=f^{-1}iew $ where, $ w\in \{b^{-1}, e^{-1}, h^{-1}\}$ so $ d_{\widehat{\Delta}}(v_i)\geq 3 $. Same as $l_{\widehat{\Delta}}(v_c)=cbe^{-1}$ and $l_{\widehat{\Delta}}(v_a)=ad$ implies $ l_{\widehat{\Delta}}(v_b)=cbd^{-1}w $ where, $ w\in \{b^{-1}, e^{-1}. h^{-1}\} $ so $ d_{\widehat{\Delta}}(v_b)\geq 3.$ Since $d_{\widehat{\Delta}}(v_h)=d_{\widehat{\Delta}}(v_a)=d_{\widehat{\Delta}}(v_d)=d_{\widehat{\Delta}}(v_f)=2$, $d_{\widehat{\Delta}}(v_e)=4, d_{\widehat{\Delta}}(v_c)=3$ and $d_{\widehat{\Delta}}(v_g)> 2$. So  $c(\widehat{\Delta})\leq c(2,2,2,2,3,3,4,4,4)\leq -\frac{\pi}{6}$. Therefore $c(\Delta)+ c(\widehat{\Delta}) \leq 0$.

\end{enumerate}
\end{enumerate}                                                                           
\end{proof}

\begin{lemma} \label{lem7} The relative presentation $\mathcal{P}= \langle G,t~|~ s(t)\rangle$ is aspherical if any one of the following holds:
\begin{enumerate}

\item $a=d^{-1}, c=f, e=b;$

\item $a=g, h=b;$

\item $a=g, f=i;$

\item $a=d, c=i;$

\item $a=d, e=b;$

\item $d=g, h=e;$

\item $a=d^{-1}, c=i, h=b;$

\item $a=d^{-1}, f=i, h=e;$

\item $a=d, c=f^{-1}, e=b;$

\item $d=g^{-1}, f=i, h=e;$

\item $a=d, c=i, h=e;$

\item $a=g^{-1}, c=f, e=b;$

\item $a=d, f=i^{-1}, e=b;$

\item $a=g, c=f, h=b;$

\item $a=d, f=i, e=b;$

\item $a=g, f=i, e=b;$

\item $a=g,f=i^{-1}, h=b;$

\item $a=d,c=f, e=b;$

\item $a=g^{-1},f=i, h=e.$
\end{enumerate}
\end{lemma}

\begin{proof} We prove the Lemma for the case  $a=d^{-1}, c=f, e=b$. The proof of the remaining cases follow similarly.

%\begin{figure}
%\includegraphics[scale=0.5]{f3.png}
%\label{case 1}
%\caption{Region $\Delta$ for case $a=d^{-1}, c=f$ and $e=b$ }
%\label{casef3}
%\end{figure}

Here $l_{\Delta}(v_b)=be^{-1}, l_{\Delta}(v_a)=ad, l_{\Delta}(v_c)=cf^{-1}, l_{\Delta}(v_d)=da, l_{\Delta}(v_f)=fc^{-1}$ and $l_{\Delta}(v_e)=eb^{-1}$ as shown in the Figure \ref{casef3(iv)}(i). There are the following two cases to consider:

\begin{figure}[H]
\centering
        \includegraphics[width=11cm]{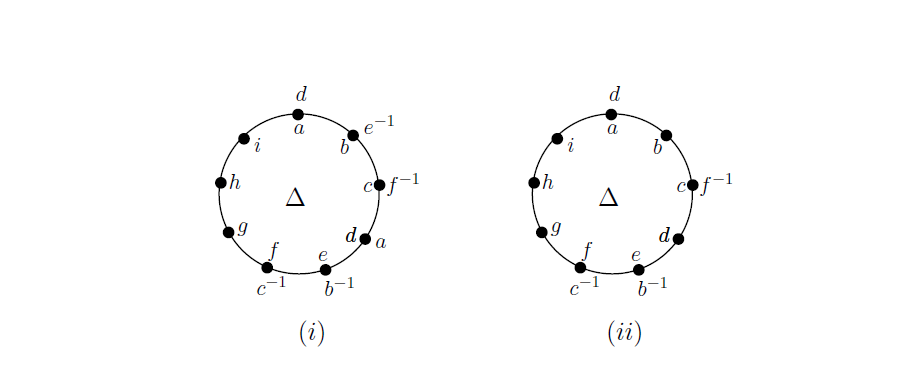}
        \caption{Regions $\Delta$ and $\widehat{\Delta}$}
    \label{casef3(iv)}
\end{figure}

\begin{enumerate} %[label=(\roman*)]

\item $d_{\Delta}(v_a)=d_{\Delta}(v_c)=d_{\Delta}(v_e)=d_{\Delta}(v_f)=2;$ 

\item $d_{\Delta}(v_b)=d_{\Delta}(v_c)=d_{\Delta}(v_e)=d_{\Delta}(v_f)=2.$

 \end{enumerate}

\begin{enumerate} %[label=(\roman*)]

\item $d_{\Delta}(v_a)=d_{\Delta}(v_c)=d_{\Delta}(v_e)=d_{\Delta}(v_f)=2.$

Here $l_{\Delta}(v_a)=ad, l_{\Delta}(v_c)=cf^{-1}, l_{\Delta}(v_f)=fc^{-1}$ and $l_{\Delta}(v_e)=eb^{-1}$ as shown in the Figure \ref{casef3(iv)}(ii). There are the following five cases to consider:

\begin{enumerate}

\item $d_{\Delta}(v_b)=d_{\Delta}(v_d)=d_{\Delta}(v_g)=d_{\Delta}(v_h)=d_{\Delta}(v_i)=3;$

\item $d_{\Delta}(v_g)>3$ only;

\item $d_{\Delta}(v_d)>3$ only;

\item $d_{\Delta}(v_h)>3$ only;

\item $d_{\Delta}(v_d)>3$ only.

\end{enumerate}

\begin{enumerate}

\begin{figure}[H]
\centering
        \includegraphics[width=7cm]{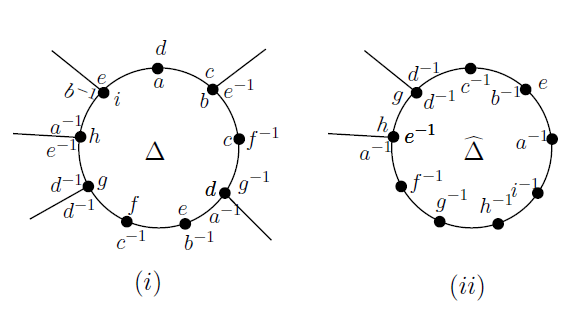}
        \caption{Regions $\Delta$ and $\widehat{\Delta}$}
    \label{casef3(iii)}
\end{figure} 

\item Here $l_{\Delta}(v_a)=ad$ and $l_{\Delta}(v_c)=cf^{-1}$ which implies $l_{\Delta}(v_b)=be^{-1}c$. Same as $l_{\Delta}(v_e)=eb^{-1}$ and $l_{\Delta}(v_c)=cf^{-1}$ which implies $l_{\Delta}(v_d)=da^{-1}g^{-1}$ as shown in the Figure \ref{casef3(iii)}(i).

Since $l_{\Delta}(v_a)=ad$ which implies  
$l_{\Delta}(v_i)=ie\{b^{-1}, h^{-1}\}$. If $l_{\Delta}(v_i)=ieb^{-1}$ which implies $l_{\Delta}(v_h)=ha^{-1}\{b^{-1}, e^{-1}\}$. If $l_{\Delta}(v_h)=ha^{-1}b^{-1}$ then $l_{\Delta}(v_g)=ga^{-1}d^{-1}$ implies $g=1$, a contradiction. If $l_{\Delta}(v_h)=ha^{-1}e^{-1}$ then $l_{\Delta}(v_g)=gd^{-2}$. So $c(\Delta)\leq \frac{\pi}{3}$. Add $c(\Delta)$ to $c(\widehat{\Delta})$ as shown in the Figure \ref{casef3(iii)}(ii). Observe that in $\widehat{\Delta}$. Here, $l_{\widehat{\Delta}}(v_{e^{-1}})=e^{-1}ha^{-1}, l_{\widehat{\Delta}}(v_{d^{-1}})=gd^{-2}$ and $l_{\widehat{\Delta}}(v_{b^{-1}})=b^{-1}e$. Since $d_{\widehat{\Delta}}(v_{b^{-1}})=2$ and all other vetices have degree atleast $3$. So $c(\widehat{\Delta}) \leq c(2, 3, 3, 3, 3, 3, 3, 3, 3)=-\frac{2\pi}{3}$. So $c(\Delta) + c(\widehat{\Delta})\leq \frac{\pi}{3} + (-\frac{2\pi}{3})= -\frac{\pi}{3}$.

If $l_{\Delta}(v_i)=ieh^{-1}$ then it solve similarly  as $l_{\Delta}(v_i)=ieb^{-1}$. 

\item $d_{\Delta}(v_g)>3$ only;

Here $l_{\Delta}(v_a)=ad$ and $l_{\Delta}(v_c)=cf^{-1}$ which implies $l_{\Delta}(v_b)=be^{-1}c$ as shown in the Figure \ref{casef3}(i). Same as $l_{\Delta}(v_e)=eb^{-1}$ and $l_{\Delta}(v_c)=cf^{-1}$ which implies $l_{\Delta}(v_d)=da^{-1}g^{-1}$. So $c(\Delta)\leq \frac{\pi}{6}$. Add $c(\Delta)$ to $c(\widehat{\Delta})$ as shown in the Figure \ref{casef3}(ii). Observe that in $\widehat{\Delta}$.

\begin{figure}[H]
\centering
        \includegraphics[width=11cm]{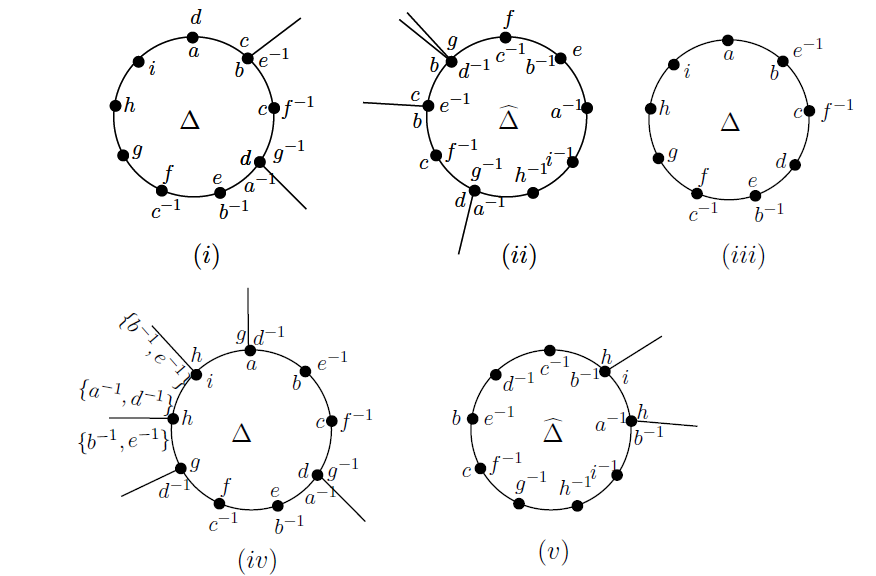}
        \caption{Regions $\Delta$ and $\widehat{\Delta}$}
    \label{casef3}
\end{figure} 
 
Here $l_{\widehat{\Delta}}(v_{e^{-1}})=e^{-1}cb$ and $l_{\widehat{\Delta}}(v_{c^{-1}})=c^{-1}f$ which implies $l_{\widehat{\Delta}}(v_{d^{-1}})=bd^{-1}gw,$ where $w\in \{b^{-1}, h^{-1}, e^{-1}\}$ so $d_{\widehat{\Delta}}(v_{d^{-1}})>3$.

Since $d_{\widehat{\Delta}}(v_{c^{-1}})=d_{\widehat{\Delta}}(v_{b^{-1}})=d_{\widehat{\Delta}}(v_{f^{-1}})=2, d_{\widehat{\Delta}}(v_{d^{-1}})>3$ and all other vertices have degree atleast $3$. So $c(\widehat{\Delta}) \leq c(2, 2, 2, 3, 3, 3, 3, 3, 4)=-\frac{\pi}{6}$. So $c(\Delta) + c(\widehat{\Delta})\leq \frac{\pi}{6} + (-\frac{\pi}{6})= 0$.

\item $d_{\Delta}(v_d)>3$ only;

\item $d_{\Delta}(v_h)>3$ only;

\item $d_{\Delta}(v_d)>3$ only.

All three cases are solved similarly as $d_{\Delta}(v_g)>3$ only.

\end{enumerate}

\item $d_{\Delta}(v_b)=d_{\Delta}(v_c)=d_{\Delta}(v_e)=d_{\Delta}(v_f)=2.$

Here $l_{\Delta}(v_e)=eb^{-1}, l_{\Delta}(v_f)=fc^{-1}, l_{\Delta}(v_b)=be^{-1}$ and $l_{\Delta}(v_c)=cf^{-1}$ as shown in the Figure \ref{casef3}(iii).
There are the following six cases to consider:
\begin{enumerate}%[label=(\roman*)]
\item $d_{\Delta}(v_d)=d_{\Delta}(v_g)=d_{\Delta}(v_h)=d_{\Delta}(v_i)=d_{\Delta}(v_a)=3$;
\item  $d_{\Delta}(v_g)>3$ only;
\item  $d_{\Delta}(v_h)>3$ only;
\item  $d_{\Delta}(v_i)>3$ only;
\item  $d_{\Delta}(v_a)>3$ only;
\item  $d_{\Delta}(v_d)>3$ only.
\end{enumerate}

%\begin{figure}[htbp]
%	\centering
%	\includegraphics[width=\textwidth]{Images/f10.PNG}
%	\caption{Region $\Delta$ for $ad, cf^{-1}$ and $eb^{-1};$}
%	\label{plot:khawar}
%\end{figure}

\begin{enumerate}
\item Here $l_{\Delta}(v_c)=cf^{-1}$ and $l_{\Delta}(v_e)=eb^{-1}$ so $l_{\Delta}(v_d)=da^{-1}g^{-1}$ as shown in the Figure \ref{casef3}(iv). Since $l_{\Delta}(v_b)= be^{-1}$ implies $l_{\Delta}(v_a)= ad^{-1}w,$ where $w\in\{a, d^{-1}, g, g^{-1}\}.$ If $l_{\Delta}(v_a)= ad^{-1}\{a, d^{-1}, g^{-1} \}$ which implies $a^3=1$ or $d^{-3}=1$ or $g^{2}=1$ a contradiction. If $l_{\Delta}(v_a)= ad^{-1}g$ implies $l_{\Delta}(v_i)= ih\{e^{-1}, b^{-1}\}$. If $l_{\Delta}(v_i)= ihb^{-1}$ then $l_{\Delta}(v_h)= ha^{-1}\{b^{-1},e^{-1}\}$. If $l_{\Delta}(v_h)= ha^{-1}b^{-1}$ then $l_{\Delta}(v_g)=ga^{-1}d^{-1}$ implies $a^2=1$, a contradiction. If $l_{\Delta}(v_h)= ha^{-1}e^{-1}$ then $l(v_g)= gd^{-2}$. So $c(\Delta)\leq \frac{\pi}{3}$. Add $c(\Delta)$ to $c(\widehat{\Delta})$ as shown in the Figure \ref{casef3}(v). Observe that in $\widehat{\Delta}$. Here $l_{\widehat{\Delta}}(v_{b^{-1}})=b^{-1}ih$ and $l_{\widehat{\Delta}}(v_{a^{-1}})=a^{-1}e^{-1}h$ so $d_{\widehat{\Delta}}(v_{e^{-1}})=d_{\widehat{\Delta}}(v_{f^{-1})}=2$ and all other edges have degree atleast $3$. So $c(\widehat{\Delta}) \leq -\frac{\pi}{3}$. So $c(\Delta) + c(\widehat{\Delta})\leq \frac{\pi}{3} + (-\frac{\pi}{3})= 0$.

%\begin{figure}
%\includegraphics[scale=0.5]{f3i.png}
%\label{case 1}
%\caption{Region $\Delta$ for case $a=d^{-1}, c=f$ and $e=b$ }
%\label{casef3(i)}
%\end{figure}

\begin{figure}[H]
\centering
        \includegraphics[width=11cm]{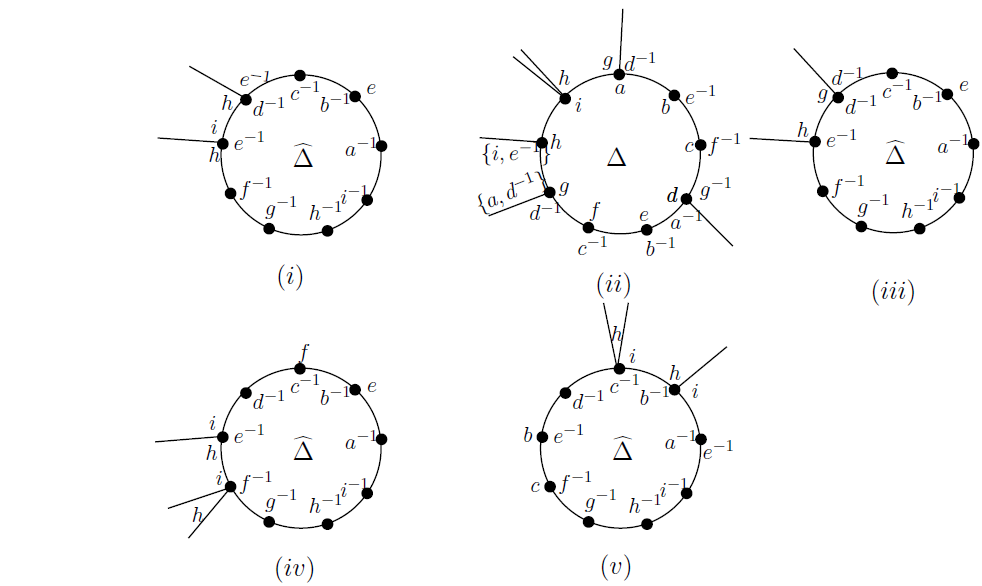}
        \caption{Regions $\Delta$ and $\widehat{\Delta}$}
    \label{casef3(i)}
\end{figure}

If $l_{\Delta}(v_i)= ihe^{-1}$ then $l_{\Delta}(v_h)= hd^{-1}\{b^{-1},e^{-1}\}$. If $l_{\Delta}(v_h)= hd^{-1}b^{-1}$ implies $l_{\Delta}(v_g)=ga^{-1}d^{-1}$ implies $a^2=1$, a contradiction, so $l_{\Delta}(v_h)= hd^{-1}e^{-1}$ which implies $l_{\Delta}(v_g)= gd^{-2}$. So $c(\Delta)\leq \frac{\pi}{3}$. Add $c(\Delta)$ to $c(\widehat{\Delta})$ as shown in the Figure \ref{casef3(i)}(i). Observe that in $\widehat{\Delta}$. Here $l_{\widehat{\Delta}}(v_{e^{-1}})=e^{-1}ih$ and $l_{\widehat{\Delta}}(v_{d^{-1}})=d^{-1}e^{-1}h$. Since $d_{\widehat{\Delta}}(v_{b^{-1}})=2$ so all other edges have degree atleast $3$. So $c(\widehat{\Delta}) \leq -\frac{2\pi}{3}$. So $c(\Delta) + c(\widehat{\Delta})\leq \frac{\pi}{3} + (-\frac{2\pi}{3})= -\frac{\pi}{3}$.

\item $d_{\Delta}(v_g)>3;$
\item $d_{\Delta}(v_h)>3.$
Both are solving similarly as the case $d_{\Delta}(v_d)=d_{\Delta}(v_g)=d_{\Delta}(v_h)=d_{\Delta}(v_i)=d_{\Delta}(v_a)=3$.

\item Here $l_{\Delta}(v_c)=cf^{-1}$ and $l_{\Delta}(v_e)=eb^{-1}$ so $l_{\Delta}(v_d)=da^{-1}g^{-1}$ as shown in the Figure \ref{casef3(i)}(ii). Since $l_{\Delta}(v_f)= fc^{-1}$ which implies $l_{\Delta}(v_g)= d^{-1}gw,$ where $w\in\{d^{-1}, g, a, a^{-1}, g\}$ which implies $g^3=1$ or $g=1$, a contradiction. So $l_{\Delta}(v_g)= d^{-1}g\{d^{-1}, a\}$. If $l_{\Delta}(v_g)= gd^{-2}$ implies $l_{\Delta}(v_h)= e^{-1}h\{a, a^{-1}, d, d^{-1}, g, g^{-1}\}$. When we continued this process we obtained $c(2, 2, 2, 2, 3, 3, 3, 3, 3, 4) \leq \frac{\pi}{6}$ as shown in the Figure \ref{casef3(i)}(ii). Add $c(\Delta)$ to $c(\widehat{\Delta})$ as shown in the Figure \ref{casef3(i)}(iii). Observe that in $\widehat{\Delta}$. Here $l_{\widehat{\Delta}}(v_{e^{-1}})=e^{-1}hw,$ where $w\in \{a, a^{-1}, d, d^{-1}, g, g^{-1}\}$ and $l_{\widehat{\Delta}}(v_{d^{-1}})=gd^{-2}$ implies $d_{\widehat{\Delta}}(v_{b^{-1}})=2$ and all other vertices have degree atleast $3$. So $c(\widehat{\Delta}) \leq -\frac{2\pi}{3}$. So $c(\Delta) + c(\widehat{\Delta})\leq \frac{\pi}{6} + (-\frac{2\pi}{3})= -\frac{\pi}{2}$.

 If $l_{\Delta}(v_g)= gad^{-1}$ implies $l_{\Delta}(v_h)= ihw,$ where $w\in \{b^{-1}, e^{-1}$. If $l_{\Delta}(v_h)=ihe^{-1}$ and since $l_{\Delta}(v_a)=ad^{-1}g$ implies $l_{\Delta}(v_i)= f^{-1}ihw$, where $w\in \{h^{-1}, b^{-1}, e^{-1}\}$. So $c(\Delta) \leq \frac{\pi}{6}$. Add $c(\Delta)$ to $c(\widehat{\Delta})$ as shown in the Figure \ref{casef3(i)}(iv). Observe that in $\widehat{\Delta}$. Here $l_{\widehat{\Delta}}(v_{e^{-1}})=e^{-1}ih$ and $l_{\widehat{\Delta}}(v_{f^{-1}})= f^{-1}ihw$, where $w\in \{h^{-1}, b^{-1}, e^{-1}\}$ implies$d_{\widehat{\Delta}}(v_{f^{-1}})>3, d_{\widehat{\Delta}}(v_{b^{-1}})=d_{\widehat{\Delta}}(v_{c^{-1})}=2$ and all other edges have degree atleast $3$. So $c(\widehat{\Delta}) \leq -\frac{\pi}{2}$. So $c(\Delta) + c(\widehat{\Delta})\leq \frac{\pi}{6} + (-\frac{\pi}{2})= -\frac{\pi}{3}$. 

 If $l_{\Delta}(v_h)=ihb^{-1}$ and since $l_{\Delta}(v_a)=ad^{-1}g$ implies $l_{\Delta}(v_i)= c^{-1}ihw$, where $w\in \{h^{-1}, b^{-1}, e^{-1}\}$. So $c(\Delta) \leq \frac{\pi}{6}$. Add $c(\Delta)$ to $c(\widehat{\Delta})$ as shown in the Figure \ref{casef3(i)}(v). Observe that in $\widehat{\Delta}$. Here $l_{\widehat{\Delta}}(v_{b^{-1}})=b^{-1}ih$ and $l_{\widehat{\Delta}}(v_{c^{-1}})= c^{-1}ihw$, where $w\in \{h^{-1}, b^{-1}, e^{-1}\}$ so $v_{c^{-1}}$ has degree greater then $3$, $d_{\widehat{\Delta}}(v_{e^{-1}})=d_{\widehat{\Delta}}(v_{f^{-1})}=2$ and all other vertices have degree atleast $3$. So $c(\widehat{\Delta}) \leq -\frac{\pi}{2}$. So $c(\Delta) + c(\widehat{\Delta})\leq \frac{\pi}{6} + (-\frac{\pi}{2})= -\frac{\pi}{3}$.

\item $d_{\Delta}(v_a)>3.$

Proved similarly as $d_{\Delta}(v_i)>3.$

\item Here $l_{\Delta}(v_b)= be^{-1}$ so $l_{\Delta}(v_a)= ad^{-1}w,$
 where $w\in\{a, d^{-1}, g, g^{-1}\}.$ If $l_{\Delta}(v_a)=ad^{-1}\{a, d^{-1}\}$ which implies $a^3=1$ or $d^{-3}=1$, a contradiction, as shown in the Figure \ref{casef3(ii)}(i). So $l_{\Delta}(v_a)= ad^{-1}w,$ where $w\in \{g, g^{-1}\}$. If $l_{\Delta}(v_a)= ad^{-1}g$ solved similarly as $d_{\Delta}(v_g)>3$. If $l_{\Delta}(a)=ad^{-1}g^{-1}$ which implies $l_{\Delta}(v_i)=if^{-1}w,$ where $w\in \{i, f^{-1}, c, c^{-1}\}$. If $l_{\Delta}(v_i)= if^{-1}c$ implies $i=1$ a contradiction. So $l_{\Delta}(v_i)=if^{-1}w,$ where $w\in \{i, f^{-1}, c^{-1}\}$ and continued this process we obtained $c(\Delta) \leq \frac{\pi}{6}$ as shown in Figure \ref{casef3(ii)}(i).
 
%\begin{figure}
%\includegraphics[scale=0.5]{f3ii.png}
%\label{case 1}
%\caption{Region $\Delta$ for case $a=d^{-1}, c=f$ and $e=b$ }
%\label{casef3(ii)}
%\end{figure}

\begin{figure}[H]
\centering
        \includegraphics[width=10cm]{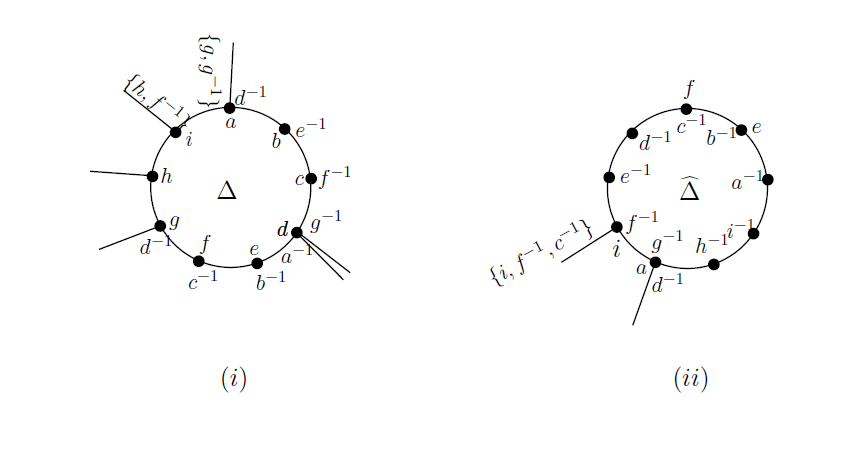}
        \caption{Regions $\Delta$ and $\widehat{\Delta}$}
    \label{casef3(ii)}
\end{figure} 

Add $c(\Delta)$ to $c(\widehat{\Delta})$ as shown in the Figure \ref{casef3(ii)}(ii). Observe that in $\widehat{\Delta}$. Here $l_{\widehat{\Delta}}(v_{g^{-1}})=g^{-1}ad^{-1},$ and $l_{\widehat{\Delta}}(v_{f^{-1}})=if^{-1}w,$ where $w\in \{i, f^{-1}, c^{-1} \}$ implies $d_{\widehat{\Delta}}(v_{b^{-1}})= d_{\widehat{\Delta}}(v_{c^{-1}})=2$ and all other vertices have degree atleast $3$. So $c(\widehat{\Delta}) \leq -\frac{\pi}{3}$. So $c(\Delta) + c(\widehat{\Delta})\leq \frac{\pi}{6} + (-\frac{\pi}{3})= -\frac{\pi}{6}$.  

\end{enumerate}

\end{enumerate}

\end{proof}

The above Lemmas \ref{lem1}-\ref{lem7} help to establish the following main result of the paper.
\begin{theorem} \label{t} The equation $$ atbtct^{-1}dtetft^{-1}gthtit^{-1}=1.$$
is solvable modulo following exceptional cases:
\begin{enumerate}

\item $d=g,f=i, h=e;$

\item $a=d, a=g, d=g$ and $R\in \{cf^{-1}, ci^{-1}, fi^{-1},  he^{-1}, hb^{-1}, eb^{-1}\};$

\item $h=e, h=b, e=b$ and $R\in \{ad^{-1}, ag^{-1}, dg^{-1}\};$

\item $a=d, a=g, d=g, h=e$ and $R\in \{cf, cf^{-1}, ci, ci^{-1}, fi, fi^{-1}\};$

\item $a=d, a=g, d=g, h=b$ and $R\in \{cf, cf^{-1}, ci, ci^{-1}, fi, fi^{-1}\};$

\item $a=d, a=g, d=g, e=b$ and $R\in \{cf, cf^{-1}, ci, ci^{-1}, fi, fi^{-1}\};$

\item $h=e, h=b, e=b, a=d^{-1}$ and $R\in \{ cf^{-1}, ci^{-1}, fi^{-1}\};$

\item $h=e, h=b, e=b, a=d$ and $R\in \{cf, cf^{-1}, ci, ci^{-1}, fi, fi^{-1}\};$

\item $h=e, h=b, e=b, a=g^{-1}$ and $R\in \{cf^{-1}, fi^{-1}\};$

\item $h=e, h=b, e=b, a=g$ and $R\in \{cf^{-1}, fi, fi^{-1}\};$

\item $h=e, h=b, e=b, d=g^{-1}, f=i;$

\item $h=e, h=b, e=b, d=g, f=i;$

\item $a=d, a=g, d=g, c=i, c=f, f=i;$

\item $a=d, a=g, d=g, h=e, h=b, e=b;$

\item $a=d, a=g, d=g, c=i, c=f, f=i$ and $R\in \{he^{-1}, eb^{-1},\};$

\item $a=d, a=g, d=g, h=e, h=b, e=b$ and $R\in \{cf, cf^{-1}, ci, ci^{-1}, fi, fi^{-1}\};$

\item $a=d, a=g, d=g, c=i, c=f, f=i, h=e, h=b, e=b.$

\end{enumerate}
\end{theorem}
\begin{remark}
\normalfont
We remark that the list of exceptional cases given in Theorem is still open. The weight test and curvature distribution can not be applied to these cases to prove  Levin's conjecture. Therefore, some new methods need to be developed to establish the validity of Levin's conjecture for this group equation of length 9.  
\end{remark}

\end{document}